\documentclass[12 pt]{article}
\usepackage{amsmath, amsthm, amssymb,verbatim,fullpage,multicol}
\usepackage{graphicx}
\usepackage{hyperref}
\usepackage[table]{xcolor}
\usepackage[T1]{fontenc}
\usepackage{appendix}
\usepackage[utf8]{inputenc}
\usepackage{braket}
\usepackage{authblk}
\usepackage{soul}
\usepackage{multirow}
\usepackage{float}
\restylefloat{table}
\newtheorem{theorem}{Theorem}[section]
\newtheorem{lemma}[theorem]{Lemma}
\newtheorem{proposition}[theorem]{Proposition}
\newtheorem{corollary}[theorem]{Corollary}
\newtheorem{definition}{Definition}[section]

\usepackage[mathscr]{euscript}
\usepackage{float}

\definecolor{lightgray}{gray}{0.85}
\definecolor{LightCyan}{rgb}{0.88,1,1}
\definecolor{midgray}{gray}{.6}

\newtheoremstyle{named}{}{}{\itshape}{}{\bfseries}{.}{.5em}{\thmnote{#3 }#1}
\theoremstyle{named}
\newtheorem*{namedtheorem}{Theorem}

\def\ZZ{{\mathbb Z}}

\newcommand{\cupdot}{\mathbin{\mathaccent\cdot\cup}}

\def\B{\mathcal{A}}
\def\C{\mathcal{A}}

\def\A{\mathcal{A}}
\def\D{\mathcal{A}}

\def\CA{\mathcal{CA}}

\def\CC{\mathcal{CA}}
\def\CD{\mathcal{CA}}

\def\a{\alpha}
\def\l{\tilde{\alpha}}
\def\r{\rho}

\title{A new characterization of the exceptional Lie algebras}
\author[1]{Pamela E. Harris\thanks{pamela.harris@usma.edu. This research was performed while the author held a National Research Council Research Associateship Award at USMA/ARL.}}
\author[2]{Erik Insko\thanks{einsko@fgcu.edu}}
\affil[1]{Department of Mathematical Sciences, United States Military Academy}
\affil[2]{Department of Mathematics, Florida Gulf Coast University}

\begin{document}
  \maketitle

\begin{abstract}
For a simple Lie algebra, over $\mathbb{C}$, we consider the weight which is the sum of all simple roots and denote it $\l$. We formally use Kostant's weight multiplicity formula to compute the ``dimension'' of the zero-weight space. 
In type $A_r$, $\l$ is the highest root, and therefore this dimension is the rank of the Lie algebra. In type $B_r$, this is the defining representation, with dimension equal to 1. In the remaining cases, the weight $\l$ is not dominant and is not the highest weight of an irreducible finite-dimensional representation. Kostant's weight multiplicity formula, in these cases, is assigning a value to a virtual representation. The point, however, is that this number is nonzero if and only if the Lie algebra is classical.  This gives rise to a new characterization of the exceptional Lie algebras as the only Lie algebras for which this value is zero. 
 
\end{abstract}

\section{Introduction}

Given a simple Lie algebra, over $\mathbb{C}$, we consider the weight $\l$ which is the sum of all simple roots, and we formally use Kostant's weight multiplicity formula to compute the ``dimension'' of the zero-weight space, which we denote by $m(\l,0)$. This representation is the adjoint representation in the Lie algebra of type $A_r$ and the defining representation in type $B_r$;
 these cases were considered by Harris in \cite{PH} and \cite{Harris}, respectively. In the remaining Lie types it is a virtual representation: a representation arising from a non-dominant integral highest weight.

Our main result in this paper is a new characterization of the exceptional Lie algebras.
\begin{namedtheorem}[Main] 
Let $\mathfrak{g}$ be a simple Lie algebra over $\mathbb{C}$ and let $\l$ denote the sum of all simple roots. Then $m(\l,0)=0$ if and only if $\mathfrak{g}$ is an exceptional Lie algebra. 
\end{namedtheorem}

In fact, we prove that if $\mathfrak{g}$ is a classical Lie algebra then $m(\l,0)$ is a positive integer. To prove our Main Theorem
 we define two subsets of the Weyl group: the \emph{Weyl alternation set} and the \emph{collapsed Weyl alternation set}, which we denote $\A({\l,0})$ and $\CA(\l,0)$ respectively. 
These subsets describe particular elements of the support of Kostant's partition function. 
We show that in the classical Lie types, the cardinality of these sets can be enumerated using only Fibonacci numbers\footnote{Let $n\in\mathbb{Z}_{\ge0}$. 
As is standard, we let $F_n$ denote the $n^{th}$ Fibonacci number, which is defined recursively by $F_n=F_{n-1}+F_{n-2}$, with $F_{0}=0$ and $F_1=1$.}. 
 Then for each respective Lie type, we provide new closed formulas for the value of Kostant's weight multiplicity formula ($m(\l,0)$) and its $q$-analog ($m_q(\l,0)$). Table \ref{mainresults} provides a summary these results.

\begin{table}

\begin{tabular}{|@{ }c@{ }|@{ }c@{ }|@{ }c@{ }|@{ }c@{ }|@{ }c@{ }|}\hline
Lie Type			&	$|\A(\l,0)|$&	$|\CA(\l,0)|$	&	$m_q(\l,0)$	&	$m(\l,0)$\\\hline\hline
$A_r$ {\footnotesize $(r\geq 1)$}	&	$F_r$&	$F_r$			&	$q+q^2+\cdots+q^{r-1}+q^r$				&$r$	\\\hline
$B_r$ {\footnotesize $(r\geq 2)$}	&	$F_{r+1}$&	$F_{r+1}$			&	$q^r$			&1	\\\hline

$C_r$ {\footnotesize$(r\geq5)$}& {\footnotesize$F_r+3F_{r-3}$}&$2F_{r-1}$& {\scriptsize$q+2q^2+3q^3+3q^4+\cdots+3q^{r-2}+2q^{r-1}+q^r$}				&$3r-6$	\\\hline

$D_r$ {\footnotesize$(r\geq7)$}& 	{\scriptsize$F_{r-1}+3F_{r-4}$}	&$2F_{r-2}$&	{\scriptsize$q+4q^2+7q^3+8q^4+8q^5+\cdots+8q^{r-3}+7q^{r-2}+4q^{r-1}+q^r$}				&$8r-24$	\\\hline

$G_2$			& 	2&0				&	0				&0	\\\hline
$F_4$			& 	4&0				&	0				&0	\\\hline
$E_6$			& 	12&0				&	0				&0	\\\hline
$E_7$			& 	18&0				&	0				&0	\\\hline
$E_8$			& 	30&0				&	0				&0	\\\hline
\end{tabular}
\label{mainresults}
\caption{Summary of main results}
\end{table}

\section{Background}
The following Lie-theoretic notation will be used throughout this manuscript, for further references see \cite{GW,Humphreys,Knapp}. Let $G$ be a simple linear algebraic group over $\mathbb C$, 
$T$ a maximal algebraic torus in $G$ of dimension $r$, and $B$ with $T\subseteq B \subseteq G$, a choice of Borel subgroup. We take $\mathfrak g$, $\mathfrak h$, and $\mathfrak b$
 to denote the Lie algebras of $G$, $T$, and $B$ respectively. We let $\Phi$ be the set of roots corresponding to $(\mathfrak {g,h})$ and $\Phi^+\subseteq\Phi$ the set of positive roots with respect to $\mathfrak b$. Let $\Delta\subseteq\Phi^+$ be the set of simple roots. We denote the set of integral and dominant integral weights by $P(\mathfrak g)$ and $P_+(\mathfrak g)$, respectively. Finally, we let $W=Norm_G(T)/T$ denote the Weyl group corresponding to $G$ and $T$, and for any $w\in W$, we let $\ell(w)$ denote the length of $w$. 
Recall that that the Weyl group is generated by the simple reflections $s_1,\ldots,s_r $, where $s_i$ denotes the reflection associated to the simple root $\alpha_i$ \cite[Chapter 9]{Humphreys}. 

The theorem of the highest weight, due to Cartan \cite{Cartan}, establishes that a finite-dimensional complex irreducible representation of $\mathfrak g$ is equivalent 
to a highest weight representation with dominant integral highest weight $\lambda$. We denote such a representation by $L(\lambda)$. 
Furthermore, to find the multiplicity of a weight $\mu$ in $L(\lambda)$, one can use Kostant's weight multiplicity formula, \cite{KMF}:
\begin{align}
m(\lambda,\mu)=\displaystyle\sum_{\sigma\in W}^{}(-1)^{\ell(\sigma)}\wp(\sigma(\lambda+\rho)-(\mu+\rho))\label{KMF},
\end{align} where $\wp$ denotes Kostant's partition function and $\rho=\frac{1}{2}\sum_{\alpha\in\Phi^+}\alpha$. Recall that Kostant's partition function is the function $\wp: \mathfrak h^* \rightarrow \ZZ_{\geq 0}$ with $\wp(\xi)=m$, where $m$ is the number of ways the weight $\xi \in \mathfrak{h}^*$ may be written as a nonnegative integral sum of positive roots.

Additionally, we define as in \cite{lusztig}, the $q$-analog of Kostant's weight multiplicity formula:
\begin{align}
m_q(\lambda,\mu)=\displaystyle\sum_{\sigma\in W}^{}(-1)^{\ell(\sigma)}\wp_q(\sigma(\lambda+\rho)-(\mu+\rho))\label{qmult},
\end{align}
where $\wp_q$ denotes the $q$-analog of Kostant's partition function. That is, for any weight $\xi\in\mathfrak{h}^*$, $\wp_q(\xi)$ is a polynomial valued function defined by
$\wp_q(\xi)=c_0+c_1q+c_2q^2+c_3q^3+\cdots+c_kq^k$,
where $c_i$ is the number of ways to write $\xi$ as a sum of exactly $i$ positive roots.

Observe that for any weight $\xi\in\mathfrak{h}^*$, $\wp_q(\xi)|_{q=1}=\wp(\xi)$. Thus, when the $q$-analog of Kostant's weight multiplicity formula is evaluated at $q=1$, we recover the original weight multiplicity formula. 
Therefore, for any weights $\lambda$ and $\mu$ in $\mathfrak{h}^*$, we have that \[ m_q(\lambda,\mu)|_{q=1}=m(\lambda,\mu). \]

In practice, when $\lambda$ and $\mu$ are fixed, it is often the case that given 
$\sigma\in W$, the value of Kostant's partition function on $\sigma(\lambda+\rho)-\rho-\mu$ evaluates to zero, i.e., $\wp(\sigma(\lambda+\rho)-\rho-\mu)=0$. With this in mind
 we define the \emph{Weyl alternation set} to be the support of Kostant's partition function:

\begin{definition}\label{definition} For $\lambda,\mu$ integral weights of $\mathfrak g$ define the \emph{Weyl alternation set (of $\lambda$ and $\mu$)} to be 
\[\mathcal A(\lambda,\mu):=\{\sigma\in W:\;\wp(\sigma(\lambda+\rho)-(\mu+\rho))>0\}.\]
\end{definition}

Even when Kostant's partition function is nonzero for an element $\sigma$ of the Weyl group, namely $\sigma\in\A(\l,0)$, 
we will encounter cases where there exists a simple transposition 
$s_i \in W$ such that $\wp(\sigma(\lambda+\rho)-\rho-\mu) = \wp(\sigma s_i(\lambda+\rho)-\rho-\mu)$. Given that the lengths of $\sigma$ and $\sigma s_i$ have opposite parity, these two terms cancel each other out in the alternating sum of Kostant's weight multiplicity formula (\ref{KMF}) and its $q$-analog (\ref{qmult}). 
For this reason, we define the \emph{collapsed Weyl alternation set} to be the subset of the support for which this cancellation does not occur:
\begin{definition}\label{def2} For $\lambda,\mu$ integral weights of $\mathfrak{g}$ define the \emph{collapsed Weyl alternation set (of $\lambda$ and $\mu$)} to be
\[\CA(\lambda,\mu):=\begin{Bmatrix}\sigma\in\A(\lambda,\mu):\mbox{ $\nexists \;s_i\in W$ such that $\wp(\sigma(\lambda+\rho)-\rho-\mu)=\wp(\sigma s_i(\lambda+\rho)-\rho-\mu)$}\end{Bmatrix}.
\]\end{definition}

Observe that Definition \ref{def2} of the collapsed Weyl alternation set $\CA(\lambda,\mu)$ further reduces the computation of $m(\lambda,\mu)$ 
(or $m_q(\lambda,\mu)$) to a smaller subset of the Weyl group than the Weyl alternation set $\A(\lambda,\mu)$. 
Namely, when computing $m(\lambda,\mu)$ (or $m_q(\lambda,\mu)$) one need only compute the value of 
Kostant's partition function (or its $q$-analog) for elements in the collapsed Weyl alternation set $\CA(\lambda,\mu)$.

Given a simple Lie algebra $\mathfrak{g}$ of rank $r$, we will describe and enumerate the sets\footnote{When we want to specify r, the rank of the Lie algebra considered, we use the notation $\A_r(\lambda,\mu)$ and $\CA_r(\lambda,\mu)$.} 
$\A_r(\l,0)$ and $\CA_r(\l,0)$ of $\l$ and $0$,
 where $\l = \sum_{\alpha \in \Delta} \alpha$ is the sum of the simple roots of $\mathfrak{g}$. 
 The sets $\CA_r(\l,0)$ will be particularly important when we consider the exceptional Lie algebras, as we will see in Section \ref{exceptional}. 
Then for any $\sigma\in\CA_r(\l,0)$, we will provide (new) closed formulas for the value of Kostant's partition function and its
 $q$-analog on the expression $\sigma(\l+\rho)-\rho$. This will allow us to provide a proof to our main result:

\begin{namedtheorem}[Main]
Let $\mathfrak{g}$ be a simple Lie algebra over $\mathbb{C}$ and let $\l$ denote the sum of all simple roots. Then $m(\l,0)=0$ if and only if $\mathfrak{g}$ is an exceptional Lie algebra. 
\end{namedtheorem}

We begin with some general background on Weyl groups acting on root spaces. Section \ref{rootspaces}, 
will develop the foundation needed for the computations in the remaining sections.
\subsection{Weyl groups acting on root spaces.} \label{rootspaces}

It is widely established throughout the literature that the finite Lie algebras are classified by the Dynkin diagrams in Figure \ref{Dynkin}, 
\cite{Humphreys}.

\begin{figure}[H]
\begin{picture}(330,240)(-15,0)
\multiput(0,215)(30,0){6}{\circle{5}}

\thinlines
\color{black}
\put(170,210){$A_n$}
\put(2,215){\line(1,0){25}}
\put(32,215){\line(1,0){25}}
\put(62,215){\line(1,0){25}}
\put(122,215){\line(1,0){25}}
\multiput(93,215)(7,0){3}{ \line(1,0){3}}

\thinlines
\put(0,205){$\alpha_1$}
\put(30,205){$\alpha_2$}
\put(60,205){$\alpha_3$}
\put(90,205){$\alpha_{4}$}
\put(120,205){$\alpha_{n-1}$}
\put(150,205){$\alpha_{n}$}

\multiput(0,170)(30,0){6}{\circle{5}}
\thinlines
\color{black}
\put(170,165){$B_n$}
\put(2,170){\line(1,0){25}}
\put(32,170){\line(1,0){25}}
\put(62,170){\line(1,0){25}}

\put(122,169){\line(1,0){26}}
\put(122,171){\line(1,0){26}}

\put(140,170){\line(-1,1){10}}
\put(140,170){\line(-1,-1){10}}
\multiput(93,170)(7,0){3}{ \line(1,0){3}}
\put(0,158){$\alpha_1$}
\put(30,158){$\alpha_2$}
\put(60,158){$\alpha_3$}
\put(90,158){$\alpha_{4}$}
\put(116,158){$\alpha_{n-1}$}
\put(150,158){$\alpha_{n}$}

\multiput(0,120)(30,0){6}{\circle{5}}
\thinlines
\color{black}
\put(170,115){$C_n$}
\put(2,120){\line(1,0){25}}
\put(32,120){\line(1,0){25}}
\put(62,120){\line(1,0){25}}
\put(122,119){\line(1,0){26}}
\put(122,121){\line(1,0){26}}
\put(135,120){\line(1,1){10}}
\put(135,120){\line(1,-1){10}}
\multiput(93,120)(7,0){3}{ \line(1,0){3}}
\thinlines
\put(0,108){$\alpha_1$}
\put(30,108){$\alpha_2$}
\put(60,108){$\alpha_3$}
\put(90,108){$\alpha_{4}$}
\put(116,108){$\alpha_{n-1}$}
\put(150,108){$\alpha_{n}$}

\multiput(0,60)(30,0){6}{\circle{5}}
\thinlines
\color{black}
\put(170,60){$D_n$}
\put(2,60){\line(1,0){25}}
\put(32,60){\line(1,0){25}}
\put(62,60){\line(1,0){25}}
\put(122,60){\line(1,0){25}}
\put(120,63){\line(0,1){20}}
\put(120,85){\circle{5}}

\put(123,82){$\alpha_{n-1}$}
\multiput(93,60)(7,0){3}{ \line(1,0){3}}

\thinlines
\put(0,50){$\alpha_1$}
\put(30,50){$\alpha_2$}
\put(60,50){$\alpha_3$}
\put(90,50){$\alpha_{4}$}
\put(120,50){$\alpha_{n-2}$}
\put(150,50){$\alpha_{n}$}

\multiput(250,170)(30,0){4}{\circle{5}}

\thicklines
\color{black}
\thinlines
\put(360,165){$F_4$}
\put(252,170){\line(1,0){25}}
\put(282,171){\line(1,0){26}}
\put(300,170){\line(-1,1){10}}
\put(300,170){\line(-1,-1){10}}
\put(282,169){\line(1,0){26}}
\put(312,170){\line(1,0){25}}
\put(250,160){$\alpha_1$}
\put(277,160){$\alpha_2$}
\put(310,160){$\alpha_3$}
\put(340,160){$\alpha_{4}$}

\multiput(250,120)(30,0){5}{\circle{5}}
\thinlines
\color{black}
\put(390,115){$E_6$}
\put(252,120){\line(1,0){25}}
\put(282,120){\line(1,0){25}}
\put(312,120){\line(1,0){25}}
\put(342,120){\line(1,0){25}}
\put(310,123){\line(0,1){20}}
\put(310,145){\circle{5}}
\put(250,110){$\alpha_1$}
\put(280,110){$\alpha_3$}
\put(310,110){$\alpha_4$}
\put(340,110){$\alpha_{5}$}
\put(370,110){$\alpha_{6}$}
\put(314,140){$\alpha_{2}$}

\multiput(250,60)(30,0){6}{\circle{5}}
\color{black}
\put(420,55){$E_7$}
\put(252,60){\line(1,0){25}}
\put(282,60){\line(1,0){25}}
\put(312,60){\line(1,0){25}}
\put(342,60){\line(1,0){25}}
\put(372,60){\line(1,0){25}}

\put(310,63){\line(0,1){20}}
\put(310,85){\circle{5}}

\put(250,50){$\alpha_1$}
\put(280,50){$\alpha_3$}
\put(310,50){$\alpha_4$}
\put(340,50){$\alpha_{5}$}
\put(370,50){$\alpha_{6}$}
\put(315,82){$\alpha_{2}$}
\put(400,50){$\alpha_{7}$}

\multiput(280,215)(30,0){2}{\circle{6}}
\color{black}
\put(330,210){$G_2$}
\put(282,215){\line(1,0){25}}
\put(282,213){\line(1,0){26}}
\put(282,217){\line(1,0){26}}
\put(300,215){\line(-1,1){10}}
\put(300,215){\line(-1,-1){10}}
\put(277,205){$\alpha_1$}
\put(310,205){$\alpha_2$}

\multiput(140,10)(30,0){7}{\circle{5}}
\put(340,5){$E_8$}
\put(142,10){\line(1,0){25}}
\put(172,10){\line(1,0){25}}
\put(202,10){\line(1,0){25}}
\put(232,10){\line(1,0){25}}
\put(262,10){\line(1,0){25}}
\put(292,10){\line(1,0){25}}

\put(200,13){\line(0,1){20}}
\put(200,35){\circle{5}}

\put(140,0){$\alpha_1$}
\put(170,0){$\alpha_3$}
\put(200,0){$\alpha_4$}
\put(230,0){$\alpha_{5}$}
\put(260,0){$\alpha_{6}$}
\put(203, 30){$\alpha_{2}$}
\put(290,0){$\alpha_{7}$}
\put(320,0){$\alpha_{8}$}
\end{picture} 
\caption{Dynkin Diagrams } \label{Dynkin}
\end{figure}
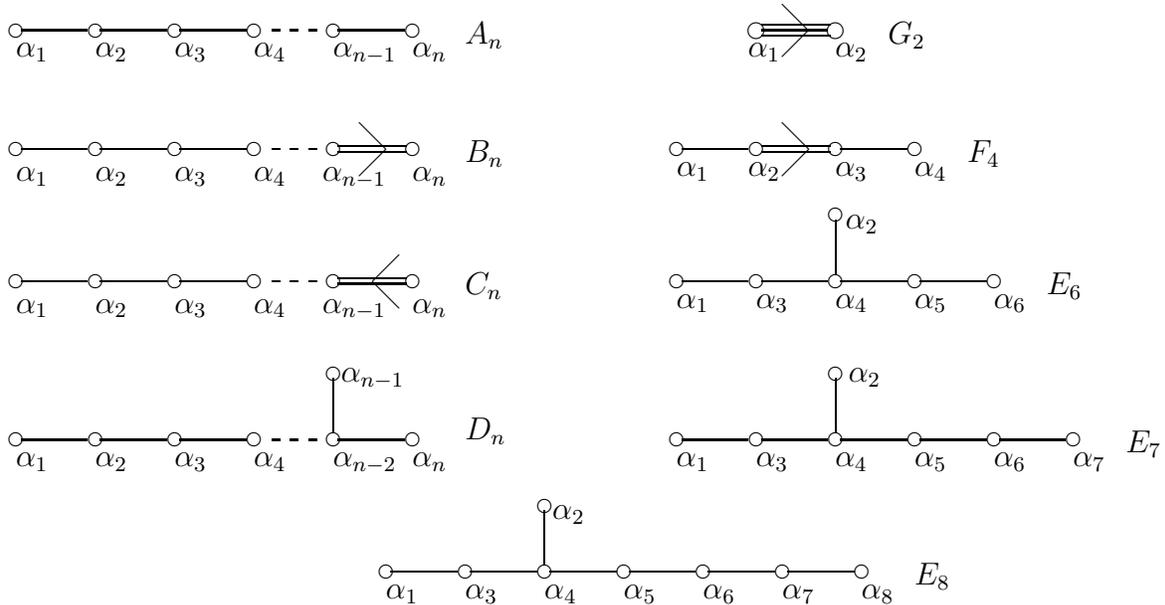

Moreover, the action of a simple reflection $s_i \in W$ on the set of simple roots $\Delta$ is also characterized by these Dynkin diagrams.
We will now recall a short synopsis of this classification and refer the reader to \cite{BB} and \cite{Car}
for more on the combinatorics of Weyl (Coxeter) groups and their actions on roots.

By the definition of a simple reflection, for any two simple roots $\alpha_i$ and $\alpha_j$ 
we have $s_i(\alpha_i) = -\alpha_i$, $s_j(\alpha_j) = -\alpha_j$, $s_{i}(\alpha_j) = \alpha_j+c_{ij} \alpha_i$ and $s_j(\alpha_i) = \alpha_i +c_{ji} \alpha_j$.
The integers $c_{ij}$ and $c_{ji}$ are in the set $\{0, 1,2,3 \}$, and 
their particular values are determined by how many 
edges (and their direction) connect the nodes $\alpha_i$ and $\alpha_j$ in the Dynkin diagrams as summarized in Figure \ref{Figure:Dynks}, \cite{Car}. 
It is then a simple exercise to show that $s_i$ permutes the remaining positive roots \cite[Lemma 4.4.3]{BB}.

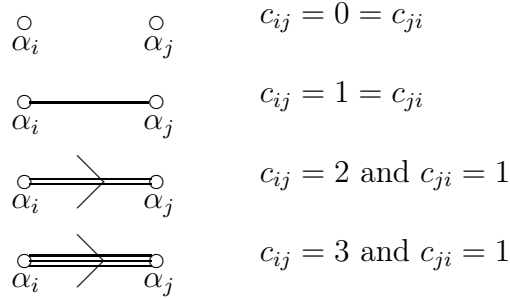
\begin{figure}[H]
 \begin{picture}(0,110)(-75,0)

\color{black}
\put (160, 100){ $c_{ij} = 0 = c_{ji}$ }
\multiput(75,100)(50,0){2}{\circle{5}}     
\put(70,90){$\alpha_i$}
\put(120,90){$\alpha_j$}

\color{black}
\multiput(75,70)(50,0){2}{\circle{5}}
\put(77,70){\line(1,0){45}}      
\put (160, 70){ $c_{ij} = 1 = c_{ji}$ }
\put(70,60){$\alpha_i$}
\put(120,60){$\alpha_j$} 

\multiput(75,40)(50,0){2}{\circle{5}}
\put(77,41){\line(1,0){46}} 
\put(105,40){\line(-1,-1){10}}
\put(105,40){\line(-1,1){10}}
\put(77,39){\line(1,0){46}}    
\put(160, 40){ $c_{ij} = 2$ and $ c_{ji} = 1$ }
\put(70,30){$\alpha_i$}
\put(120,30){$\alpha_j$}

\multiput(75,10)(50,0){2}{\circle{5}}
\put(77,12){\line(1,0){46}} 
\put(105,10){\line(-1,-1){10}}
\put(105,10){\line(-1,1){10}}
\put(77,8){\line(1,0){46}} 
\put(77,10){\line(1,0){46}}    
\put (160, 10){ $c_{ij} = 3$ and $ c_{ji} = 1$ }
\put(70,0){$\alpha_i$}
\put(120,0){$\alpha_j$}

\end{picture}
\caption{Three cases for the values of $c_{ij}$} \label{Figure:Dynks}
\end{figure}

\section{$\A(\l,0)$ for the classical Lie algebras}

Before identifying the alternation sets $\A(\l,0)$ for each classical Lie type, we classify which elements $\sigma \in W$ are never in any alternation set $\A(\alpha,0)$, where $\alpha$ is a positive root.
Lemmas 3.3 and 3.4 in \cite{HIW} allow us to identify a set of Weyl group elements which are not in the Weyl alternation set $\A(\alpha,0)$ for any root $\alpha \in \Phi^+$ and any classical Lie type.  
\begin{lemma}\label{classic}
Let $\sigma \in W$ be a Weyl group element in any classical Lie type.  If $\sigma$ contains a subword of the form 
\[ s_i s_{i+1} s_{i+2}, \ 
s_{i+2} s_{i+1} s_{i}, \text{ or }
s_{i+1}s_{i}s_{i+2}  \]
or any product of four consecutive simple reflections $s_i, s_{i+1}, s_{i+2}, s_{i+3}$ (in any order) then $\sigma$ is not in the Weyl alternation set $\A(\alpha,0)$ for any 
$\alpha \in \Phi^+$.
\end{lemma}
\begin{proof}
In the classical Lie types, the highest root $\overline{\alpha}$ is of the form $ \displaystyle \overline{\alpha} = \sum_{\alpha_i \in \Delta} c_i\alpha_i$ with coefficients $c_i \in \{1,2\}$.
Every $\sigma \in W$ permutes the roots in $\Delta$, and so for any root $\alpha$, we have $\sigma (\alpha)  \leq \overline{\alpha}$ in the weight lattice.  Thus the weight $ \displaystyle \sigma(\alpha) = 
\sum_{\alpha_i \in \Delta} d_i \alpha_i$ has coefficients $d_i$ from the set $\{ -2,-1,0,1,2 \}$.  
It follows that if an element $\sigma \in W$ acts on the weight $\rho$ by $\sigma (\rho) = \rho- c \alpha_i$ with $c \geq 3$ for some simple root $\alpha_i$, then $\sigma (\alpha + \rho) -\rho < \alpha - 3 \alpha_i$ will have a negative coefficient when written as a linear combination of positive simple roots.  
Lemmas 3.3 and 3.4 in \cite{HIW} showed that any $\sigma \in W$ containing one of the subwords listed above acts on the weight
 $\rho$ by \[ \sigma(\rho) = \rho - \left (\sum c_i \alpha_i \right ) \text{ with one of the } c_i \geq 3 .\]  Hence any $\sigma$ containing
one of the listed subwords is not in the alternation set $\A(\alpha, 0)$ for any positive root $\alpha$.
\end{proof}
Next we classify how the simple reflections $s_i$ act on the sum of simple roots $\l $.
\begin{lemma}\label{lambda}
Let $\l=\displaystyle\sum_{\a\in\Delta}\a=\a_1+\cdots+\a_r$. Then the simple root reflections act on $\l$ in the following way:
\begin{itemize}
\item Type $A_r$: $s_1(\l)=\l-\a_1$ and $s_r(\l)=\l-\a_r$
\item Type $B_r$: $s_1(\l)=\l-\a_1$
\item Type $C_r$: $s_1(\l)=\l-\a_1$, $s_{r-1}(\l)=\l+\a_{r-1}$, and $s_r(\l)=\l-\a_r$ 
\item Type $D_r$: $s_1(\l)=\l-\a_1$, $s_{r-2}(\l)=\l+\a_{r-2}$, $s_{r-1}(\l)=\l-\a_{r-1}$, and $s_r(\l)=\l-\a_r$ 
\end{itemize}
For all other cases: $s_i(\l)=\l.$
\end{lemma}
\begin{proof}
Type $A_r$: For $1\leq i\leq r$, $s_i(\a_i)=-\a_i$, and $s_i(\a_{j})=s_j(\a_i)=\a_{i}+\a_j$, whenever $i$ and $j$ are consecutive integers. Hence
\begin{align*}
s_1(\l)&=-\a_1+(\a_1+\a_2)+\a_3+\cdots+\a_r=\l-\a_1,\\ s_r(\l)&=\a_1+\cdots+\a_{r-2}+(\a_{r-1}+\a_r)-\a_r=\l-\a_r.\end{align*}
For $2\leq i\leq r-1$ observe that \[s_i(\l)=\a_1+\cdots+\a_{i-2}+(\a_{i-1}+\a_i)-\a_i+(\a_i+\a_{i+1})+\a_{i+2}+\cdots+\a_r=\l.\]
\noindent
Type $B_r$: For $1\leq i\leq r-1$, $s_i(\a_i)=-\a_i$, $s_i(\a_{i-1})=\a_{i-1}+\a_i$, $s_i(\a_{i+1})=\a_i+\a_{i+1}$. If $i=r$, then $s_r(\a_r)=-\a_r$ and $s_r(\a_{r-1})=\a_{r-1}+2\a_r.$ Notice that
\begin{align*}
s_1(\l)&=-\a_1+(\a_1+\a_2)+\a_3+\cdots+\a_r=\a_2+\cdots+\a_r=\l-\a_1,\\
s_2(\l)&=(\a_1+\a_2)-\a_2+(\a_2+\a_3)+\a_4+\cdots+\a_r=\l.
\end{align*}
For $2\leq i\leq r-1$,
\[s_i(\l)=\a_1+\cdots+\a_{i-2}+(\a_{i-1}+\a_i)-\a_i+(\a_i+\a_{i+1})+\a_{i+2}+\cdots+\a_r=\l.\]
Finally,
\[s_r(\l)=\a_1+\cdots+\a_{r-2}+(\a_{r-1}+2\a_r)-\a_r=\l.\]
\noindent
Type $C_r$: For $1 \leq i \leq r$ we have $s_i(\alpha_i) = -\alpha_i $, $ s_i (\alpha_{i-1}) = \alpha_{i-1}+ \alpha_i$. For $1\leq i\leq r-2$, $ s_i (\alpha_{i+1}) = \alpha_i + \alpha_{i+1},$ while $s_{r-1} (\alpha_{r}) = 2\alpha_{r-1}+ \alpha_r .$ 
Observe that
\[s_1(\l)=-\a_1+(\a_1+\a_2)+\a_3+\cdots+\a_r=\l-\a_1.\]
For $2\leq i\leq r-2$, 
\[s_i(\l)=\a_1+\cdots+\a_{i-2}+(\a_{i-2}+\a_i)-\a_i+(\a_i+\a_{i+1})+\a_{i+2}+\cdots+\a_r=\l.\]
Finally observe that
\begin{align*}s_{r-1}(\l)&=\a_1+\cdots+\a_{r-3}+(\a_{r-2}+\a_{r-1})-\a_{r-1}+(2\a_{r-1}+\a_r)=\l+\a_{r-1}\mbox{, and }\\
s_r(\l)&=\a_1+\cdots+\a_{r-2}+(\a_{r-1}+\a_r)-\a_r=\l-\a_r.
\end{align*}
\noindent
Type $D_r$: For $1\leq i\leq r$, $s_i(\alpha_i)=-\alpha_i$. If $1\leq i<j\leq r-1$ with $|i-j|=1$ or if $i=r-2$ and $j=r$, then $s_i(\alpha_j)=s_j(\alpha_i)=\alpha_i+\alpha_j$.
For $i=r-1$ or $i=r$ we have that $s_{r-1}(\alpha_{r})=\alpha_r$ and $s_{r}(\alpha_{r-1})=\alpha_{r-1}.$
Observe that
\begin{align*}
s_1(\l)&=-\a_1+(\a_1+\a_2)+\a_3+\cdots+\a_r=\l-\a_1\mbox{, and}\\
s_2(\l)&=(\a_1+\a_2)-\a_2+(\a_2+\a_3)+\a_4+\cdots+\a_r=\l.
\end{align*}
For $3\leq i\leq r-3$,
\[s_i(\l)=\a_1+\cdots+\a_{i-2}+(\a_{i-1}+\a_i)-\a_i+(\a_i+\a_{i+1})+\a_{i+2}+\cdots+\a_r=\l.\]
Finally observe that
\begin{align*}
s_{r-1}(\l)&=\a_1+\cdots+\a_{r-3}+(\a_{r-2}+\a_{r-1})-\a_{r-1}+\a_r=\l-\a_{r-1}\mbox{, and}\\
s_r(\l)&=\a_1+\cdots+\a_{r-3}+(\a_{r-2}+\a_r)+\a_{r-1}-\a_r=\l-\a_r.
\end{align*}
\end{proof}

Given the above results we now proceed through an analysis of $\A(\l,0)$ for each classical Lie algebra.

\subsection{Types $A_r$ and $B_r$}\label{typeA}
In this section we consider the Lie algebras of type $A_r$ and $B_r$, that is $\mathfrak{sl}_{r+1}(\mathbb{C})$ and $\mathfrak{so}_{2r+1}(\mathbb{C})$, respectively. 
We consider the representation with highest weight $\l=\a_1+\cdots+\a_r$ in each case. This representation is the adjoint representation in Lie type $A_r$, and it is the defining representation in Lie type $B_r$. In these cases we have the following results.
\begin{theorem}\label{p1}
Let $r\geq 1$ and $\l$ denote the highest root of $\mathfrak{sl}_{r+1}(\mathbb{C})$. Then $\sigma\in\A_r(\l,0)$ if and only if $\sigma=s_{i_1}s_{i_2}\cdots s_{i_k}$, where $i_1,\ldots,i_k$ are non consecutive integers between 2 and $r-1$.
\end{theorem}

\begin{corollary}\label{p2}
If $r\geq 1$ and $\l$ is the highest root of $\mathfrak{sl}_{r+1}(\mathbb{C})$, then $|\A_r(\l,0)|=F_r.$ 
\end{corollary}

The above results are Theorems 1.2 and 2.1 in \cite{PH}. For type $B_r$ we have the following:

\begin{theorem}\label{TypeB} Let $r\geq 2$ and $\l$ denote the sum of the simple roots of $\mathfrak{so}_{2r+1}(\mathbb{C})$. Then $\sigma\in\B_r(\l,0)$ if and only if $\sigma=1$ or $\sigma=s_{i_1}s_{i_2}\cdots s_{i_j}$ for some nonconsecutive integers $i_1,\ldots,i_j$ between 2 and $r$.
\end{theorem}

\begin{corollary}\label{p3}
If $r\geq 2$ and $\l$ denotes the sum of the simple roots of $\mathfrak{so}_{2r+1}(\mathbb{C})$, then $|\B_r(\lambda,0)|=F_{r+1}$. 
\end{corollary}

Theorem~\ref{TypeB} and Corollary \ref{p3} are Theorem 2.1 and Theorem 1.1 in \cite{Harris}, respectively.

\subsection{Type $C_r$}\label{typeC}


In this section we consider the Lie algebra of type $C_r$, $\mathfrak{sp}_{2r}(\mathbb{C})$, the symplectic Lie algebra of rank $r$. We begin by listing some elements of $W_r$, 
the Weyl group of $\mathfrak{sp}_{2r}$, which are not in $\A_r(\l,0)$, where $\l=\a_1+\cdots+\a_r.$

Notice that by Lemma 2.2 in \cite{HIW} and Lemma \ref{lambda} we can see that for $2\leq i\leq r-3$
\[s_is_{i+1}(\l+\r)-\r=\l+\r-4\a_i-\a_{i+1}-\r=\l-2\a_i-\a_{i+1},\]
\[s_{i+1}s_i(\l+\r)-\r=\l+\r-2\a_i-2\a_{i+1}-\r=\l-\a_i-2\a_{i+1},\]
and
\[s_is_{i+1}s_i(\l+\r)-\r=\l+\r-2\a_i-2\a_{i+1}-\r=\l-2\a_i-2\a_{i+1}.\]
Hence $s_{i+1}s_i$ and $s_is_{i+1}$ are not in $\C_r(\l,0)$.

By Lemma 2.6 in \cite{HIW}, when $i=r-1$ we have that
\[s_{r-1}s_r(\l+\r)-\r=\l+\a_{r-1}-(2\a_{r-1}+\a_r)+\r-3\a_{r-1}-\a_{r}-\r=\l-4\a_{r-1}-2\a_r,\]
and
\[s_rs_{r-1}(\l+\r)-\r=\l+\r-\a_{r-1}-3\a_{r}-\r=\l-\a_{r-1}-3\a_r.\]
Hence $s_{r-1}s_r$ and $s_rs_{r-1}$ are not in $\C_r(\l,0)$.

These statements imply that $\C_r(\l,0)$ cannot contain any elements with consecutive factors, other than those listed below.

\begin{proposition}\label{setC}
The following elements of $W_r$ are in $\C_r(\l,0)$
\begin{itemize}
\item $(r\geq 2)$: 1, i.e. the identity element of $W_r$
\item $(r\geq 3)$: $s_i$ for any $2\leq i\leq r-1$
\item $(r\geq 4)$: $s_{r-2}s_{r-1}$, $s_{r-1}s_{r-2}$ and $s_{r-2}s_{r-1}s_{r-2}$.
\end{itemize}
\end{proposition}
\begin{proof}
Since $1(\l+\r)-\r=\a_1+\cdots+\a_r$, it follows that $1\in\C_r(\l,0)$.
By Lemma 2.2 in \cite{HIW} we know exactly how a simple root reflection acts on $\r$. Then using Lemma~\ref{lambda} we have that for $2\leq i\leq r-2$
\[s_i(\l+\r)-\r=\l+(\r-\a_i)-\r=\a_1+\cdots+\a_{i-1}+\a_{i+1}+\cdots+\a_r\]
and \[s_{r-1}(\l+\r)-\r=\l+\a_{r-1}+\r-\a_{r-1}-\r=\l.\]
Thus $s_i\in\C_r(\l,0)$ whenever $2\leq i\leq r-1.$
Observe that 
\begin{align*}
s_{r-2}s_{r-1}(\l+\r)-\r&=\l+(\a_{r-2}+\a_{r-1})+\r-2\a_{r-2}-\a_{r-1}-\r=\l-\a_{r-2},\\
s_{r-1}s_{r-2}(\l+\r)-\r&=\l+\a_{r-1}+\r-\a_{r-2}-2\a_{r-1}-\r=\l-\a_{r-2}-\a_{r-1},\mbox{ and}\\
s_{r-2}s_{r-1}s_{r-2}(\l+\r)-\r&=\l+(\a_{r-1}+\a_r)+\r-2\a_{r-1}-2\a_r-\r=\l-\a_{r-2}-\a_{r-1}.
\end{align*}
Thus $s_{r-2}s_{r-1},\ s_{r-1}s_{r-2}, \ s_{r-2}s_{r-1}s_{r-2}$ are in $\C_r(\l,0)$, whenever $r\geq 4.$
\end{proof}

We call the subwords described in Proposition \ref{setC} the \emph{basic allowable subwords} of Type $C_r$.

\begin{theorem} Let $\sigma\in W$. Then $\sigma\in\C_r(\l,0)$ if and only if $\sigma=1$ or if $\sigma$ is a commuting product of the basic allowable subwords of Type $C_r$.
\end{theorem}

We now defined the following recurrence relation:
\begin{definition}

Let $n\geq 2$. Then the sequence $\tilde{F}_n$ of positive integers is defined by the recurrence relation
\begin{align}\tilde{F}_n=\tilde{F}_{n-1}+\tilde{F}_{n-2},\label{defFtilde}
\end{align}
where $\tilde{F}_1=2, \tilde{F}_2=6.$
\end{definition}

It is easy to see that if $F_0=0$, then for $n\geq 1$ we have that $\tilde{F}_n=F_{n+2}+3F_{n-1}$.

\begin{theorem}\label{cardC}Let $r\geq 3$. Then $|\C_r(\l,0)|=\tilde{F}_{r-2}.$
\end{theorem}

\begin{proof}
We begin by defining a family of subsets of $\C_r(\l,0)$. For $k\geq 2$ let 
\begin{align}
N_{k}&:=\begin{Bmatrix}\sigma\in\D_r(\l,0):\begin{matrix}\mbox{ $\sigma=s_{i_1}s_{i_2}\cdots s_{i_j}$ for some nonconsecutive}\\\mbox{integers $i_1, i_2, \ldots, i_j$ between 2 and $k$} \end{matrix}\end{Bmatrix}.\label{N_k}
\end{align}
If $k<2$, then let $N_k=\emptyset$. It is a standard combinatorial argument to show that $|N_{k}|=F_{k+1}$, whenever $k\geq 2$.
Therefore if we let 
\begin{align*}
N_{r-4}(s_{r-1}s_{r-2})&=\{\sigma s_{r-1}s_{r-2}: \sigma\in N_{r-4}\}\\
N_{r-4}(s_{r-2}s_{r-1})&=\{\sigma s_{r-2}s_{r-1}: \sigma\in N_{r-4}\}\\
N_{r-4}(s_{r-2}s_{r-1}s_{r-2})&=\{\sigma s_{r-2}s_{r-1}s_{r-2}: \sigma\in N_{r-4}\},
\end{align*}
we have that \begin{align}
\C_r(\l,0)&=N_{r-1}\;\cupdot\; N_{r-4}(s_{r-1}s_{r-2})\;\cupdot\; N_{r-4}(s_{r-2}s_{r-1})\;\cupdot\; N_{r-4}(s_{r-2}s_{r-1}s_{r-2}).\label{unionC}\end{align} 
Thus $|\C_r(\l,0)|= |N_{r-1}|+3|N_{r-4}|=F_r+3F_{r-3}=\tilde{F}_{r-2}. $
\end{proof}

\subsection{Type $D_r$}\label{typeD}

In this section we consider the Lie algebra of type $D_r$, $\mathfrak{so}_{2r}(\mathbb{C})$, the special orthogonal Lie algebra. We begin by listing some elements of $W_r$, 
the Weyl group of $\mathfrak{so}_{2r}(\mathbb{C})$, which are not in $\A_r(\l,0)$, where $\l=\a_1+\cdots+\a_r.$

\begin{lemma} \label{D_forbidden_subwords}
The following words are never subwords of any element in the Weyl alternation set $\A_r(\l,0)$ in type $D_r$:
\begin{itemize}
 \item $s_1, s_{r-1}, s_r$
 \item $s_is_{i+1}$, $s_{i+1}s_i$ for $2\leq i\leq r-4$ 
 \item $s_{r-2}s_{r-1}$, $s_{r-1}s_{r-2}$, $s_{r}s_{r-2}$, and $s_{r-2}s_r$ 
 \item $s_{r-2}s_{r-1}s_{r-2}$ and $s_{r-2}s_r s_{r-2}$.
\end{itemize}
\end{lemma}

\begin{proof}
Notice that by Lemma 2.2 in \cite{HIW} and Lemma \ref{lambda} we can see that
\[s_1(\l+\r)-\l=\l-\a_1+\r-\a_1-\r=\l-2\a_1,\]
\[s_{r-1}(\l+\r)-\r=\l-\a_{r-1}+\r-\a_{r-1}-\r=\l-2\a_{r-1},\]
and
\[s_r(\l+\r)-\r=\l-\a_r+\r-\a_r-\r=\l-2\a_r.\]
Hence $s_1$ and $s_r$ are not in $\D_r(\l,0)$.

If $2\leq i\leq r-4$, then
\begin{align*}
s_is_{i+1}(\l+\r)-\r&=\l+\r-2\a_{i}-\a_{i+1}-\r=\l-2\a_{i}-\a_{i+1},\\
s_{i+1}s_i(\l+\r)-\r&=\l+\r-\a_i-2\a_{i+1}-\r=\l-\a_i-2\a_{i+1}.
\end{align*}
Hence $s_is_{i+1}$ and $s_{i+1}s_i$ are not in $\D_r(\l,0)$ for any $2\leq i\leq r-4$.
Also notice that
\begin{align*}
s_{r-2}s_{r-1}(\l+\r)-\r&=\l+\a_{r-2}-(\a_{r-2}+\a_{r-1})+\r-2\a_{r-2}-\a_{r-1}-\r=\l-2\a_{r-2}-2\a_{r-1},\\
s_{r-1}s_{r-2}(\l+\r)-\r&=\l -\a_{r-1}+(\a_{r-2}+\a_{r-1})+\r-\a_{r-2}-2\a_{r-1}-\r=\l-2\a_{r-1},\\
s_{r}s_{r-2}(\l+\r)-\r&=\l-\a_r+(\a_{r-2}+\a_r)+\r-\a_{r-2}-2\a_r-\r=\l-2\a_r,\\
s_{r-2}s_r(\l+\r)-\r&=\l+\a_{r-2}+(\a_{r-2}+a_{r})+\r-2\a_{r-2}-\a_r-\r=\l-2\a_{r-2}-2\a_r.
\end{align*}
Thus $s_{r-2}s_{r-1}$, $s_{r-1}s_{r-2}$, $s_{r}s_{r-2}$, and $s_{r-2}s_r$ are not in $\D_r(\l,0)$.
If $2\leq i\leq r-4$, then
\[s_is_{i+1}s_i(\l+\r)-\r=\l+\r-2\a_i-2\a_{i+1}-\r=\l-2\a_i-2\a_{i+1}.\]
Hence $s_is_{i+1}s_i$ is not in $\D_r(\l,0)$ whenever $2\leq i\leq r-4$.
Finally notice that
\[s_{r-2}s_{r-1}s_{r-2}(\l+\r)-\r=\l+\r-2\a_{r-2}-2\a_{r-1}-\r=\l-2\a_{r-2}-2\a_{r-1}\]
and 
\begin{align*}s_{r-2}s_r s_{r-2}(\l+\r)-\r&
=\l+\a_{r-2}+(\a_{r-2}+\a_r)-\a_{r-2}+(\a_{r-2}+\a_{r})+\r-2\a_{r-2}-2\a_r-\r\\&=\l-2\a_{r-2}-2\a_{r}.\end{align*}

Hence $s_{r-2}s_{r-1}s_{r-2}$ and $s_{r-2}s_r s_{r-2}$ are not in $\D_r(\l,0)$.
\end{proof}

Lemma \ref{D_forbidden_subwords} implies that $\D_r(\l,0)$ cannot contain any elements with consecutive factors, other than those listed below.

\begin{proposition}\label{setD}
The following elements of $W_r$ are in $\D_r(\l,0)$
\begin{itemize}
\item $(r\geq 2)$: 1, i.e. the identity element of $W_r$
\item $(r\geq 4)$: $s_i$ for any $2\leq i\leq r-2$
\item $(r\geq 5)$: $s_{r-3}s_{r-2}$, $s_{r-2}s_{r-3}$, and $s_{r-3}s_{r-2}s_{r-3}$
\end{itemize}
\end{proposition}
\begin{proof}
Since $1(\l+\r)-\r=\a_1+\cdots+\a_r$, it follows that $1\in\D_r(\l,0)$.
By Lemma 2.2 in \cite{HIW} we know exactly how a simple root reflection acts on $\r$. Then using Lemma~\ref{lambda} we have that for $2\leq i\leq r-3$
\[s_i(\l+\r)-\r=\l+(\r-\a_i)-\r=\l-\a_i\]
and \[s_{r-2}(\l+\r)-\r=\l+\a_{r-2}+\r-\a_{r-2}-\r=\l.\]
Thus $s_i\in\D_r(\l,0)$ whenever $2\leq i\leq r-2.$
Observe that 
\begin{align*}
s_{r-3}s_{r-2}(\l+\r)-\r&=\l+(\a_{r-3}+\a_{r-2})+\r-2\a_{r-3}-\a_{r-2}-\r=\l-\a_{r-3},\\
s_{r-2}s_{r-3}(\l+\r)-\r&=\l+\a_{r-2}+\r-\a_{r-3}-2\a_{r-2}-\r=\l-\a_{r-3}-\a_{r-2},\\
s_{r-3}s_{r-2}s_{r-3}(\l+\r)-\r&=\l+(\a_{r-3}+\a_{r-2})+\r-2\a_{r-3}-2\a_{r-2}-\r=\l-\a_{r-3}-\a_{r-2},\\
\end{align*}
hence $s_{r-3}s_{r-2}, s_{r-2}s_{r-3},s_{r-3}s_{r-2}s_{r-3}\in\D_r(\l,0)$.
\end{proof}

\begin{theorem} Let $\sigma\in W$. Then $\sigma\in\D_r(\l,0)$ if and only if $\sigma=1$ or if $\sigma$ is a commuting product of the basic allowable subwords of Type $D_r$.
\end{theorem}

\begin{theorem}\label{cardD}Let $r\geq 4$. Then $|\D_r(\l,0)|=\tilde{F}_{r-3}.$
\end{theorem}

\begin{proof}
We begin by defining a family of subsets $N_k$ of $\D_r(\l,0)$as in the proof of Theorem~\ref{cardC}. For $k\geq 2$ let \begin{align}
N_{k}&=\begin{Bmatrix}\sigma\in\D_r(\l,0):\begin{matrix}\mbox{ $\sigma=s_{i_1}s_{i_2}\cdots s_{i_j}$ for some nonconsecutive}\\\mbox{integers $i_1, i_2, \ldots, i_j$ between 2 and $k$} \end{matrix}\label{N_kD}\end{Bmatrix}.
\end{align}
If $k<2$, then let $N_k=\emptyset$. We again note that $|N_{k}|=F_{k+1}$, whenever $k\geq 2$.
Therefore if we let 
\begin{align*}
N_{r-5}(s_{r-3}s_{r-2})&=\{\sigma s_{r-3}s_{r-2}: \sigma\in N_{r-5}\}\\
N_{r-5}(s_{r-2}s_{r-3})&=\{\sigma s_{r-2}s_{r-3}: \sigma\in N_{r-5}\}\\
N_{r-5}(s_{r-3}s_{r-2}s_{r-3})&=\{\sigma s_{r-3}s_{r-2}s_{r-3}: \sigma\in N_{r-5}\}
\end{align*} 
then \begin{align}\D_r(\l,0)&=N_{r-2}\;\cupdot\; N_{r-5}(s_{r-3}s_{r-2})\cupdot N_{r-5}(s_{r-2}s_{r-3})\cupdot N_{r-5}(s_{r-3}s_{r-2}s_{r-3}).\label{unionD}\end{align}
Thus $|\D_r(\l,0)|= |N_{r-2}|+3|N_{r-5}|=F_{r-1}+3F_{r-4}=\tilde{F}_{r-3}. $
\end{proof}

\section{$\CA(\l,0)$ for the classical Lie algebras}\label{collapsed}
We now describe and enumerate the collapsed Weyl alternation sets $\CA(\l,0)$ for the classical Lie algebras, where $\l=\a_1+\cdots+\a_r$ is the sum of all simple roots. 
Recall that the collapsed Weyl alternation set allows us to further reduce the computation for finding the value of $m(\l,0)$ and $m_q(\l,0)$, 
since for any element  $\sigma\in\A(\l,0)\setminus\CA(\l,0)$ there exists a simple root reflection $s_i$ such that $\sigma(\l+\r)-\r=\sigma s_i(\l+\r)-\r$. 
Thus $\wp(\sigma(\l+\r)-\r)=\wp(\sigma s_i(\l+\r)-\r)$ and similarly $\wp_q(\sigma(\l+\r)-\r)=\wp_q(\sigma s_i(\l+\r)-\r)$.  
Since $\ell(\sigma)=\ell(\sigma s_i)\pm 1$, we know that their corresponding terms in Kostant's weight multiplicity (and its $q$-analog) have opposite signs. 
Thus the sum of these two terms will be zero. 

We begin by observing the following results for the classical Lie algebras.
\begin{theorem}\label{triA}
If $r\geq 1$ and $\mathfrak{g}=\mathfrak{sl}_{r+1}(\mathbb{C})$, then $|\CA(\l,0)|=F_r.$
\end{theorem}
\begin{theorem}\label{triB}
If $r\geq 2$ and $\mathfrak{g}=\mathfrak{so}_{2r+1}(\mathbb{C})$, then $|\CA(\l,0)|=F_{r+1}$.
\end{theorem}
Theorems \ref{triA} and \ref{triB} follow from the fact that for any $\sigma\in\A(\l,0)$ and for any $1\leq i\leq r$, $\sigma(\l+\rho)-\rho\neq \sigma s_i(\l+\rho)-\rho$. Thus $\A(\l,0)=\CA(\l,0)$ and the theorems hold.

\begin{theorem}\label{triC} 
If $r\geq 4$ and $\mathfrak{g}=\mathfrak{sp}_{2r}(\mathbb{C})$, then $|\CC_r(\l,0)|=2F_{r-1}$.
\end{theorem}

\begin{proof}
In Section~\ref{typeC} we noticed that for Lie type $C_r$ we have 
\begin{align}
s_{r-1}s_{r-2}(\l+\r)-\r&=s_{r-2}s_{r-1}s_{r-2}(\l+\r)-\r.\label{eq1typec}
\end{align}
Now we note that $s_{r-2} s_{r-1}s_{r-2}=s_{r-1}s_{r-2}s_{r-1}$. Hence Equations (\ref{unionC}) and (\ref{eq1typec}) along with the definition of $\CA(\l,0)$ imply that the collapsed Weyl alternation set $\CC_r(\l,0)$ for Lie type $C_r$ ($r\geq 4$) is given by
\begin{align}\CC_r(\l,0)=N_{r-1}\cupdot N_{r-4}(s_{r-2}s_{r-1}).\label{collapseC}\end{align}
The theorem follows directly from Equations (\ref{collapseC}) and the fact that $|N_k|=F_{k+1}$, whenever $k\geq 2$.

\end{proof}

\begin{theorem}\label{triD} If $r\geq 5$ and $\mathfrak{so}_{2r}(\mathbb{C})$, then  $|\CD_r(\l,0)|=2F_{r-2}$.
\end{theorem}

\begin{proof}
In a similar way we recall that in Section~\ref{typeD} for Lie type $D_r$ we have
\begin{align}
s_{r-2}s_{r-3}(\l+\r)-\r&=s_{r-3}s_{r-2}s_{r-3}(\l+\r)-\r.\label{eq1typed}
\end{align}
Now we note that $s_{r-3} s_{r-2}s_{r-3}=s_{r-2}s_{r-3}s_{r-2}$. Hence Equations Equations (\ref{unionD}) and (\ref{eq1typed}) along with the definition of $\CA(\l,0)$ imply that the collapsed Weyl alternation set $\CD_r(\l,0)$ for Lie type $D_r$ ($r\geq 5$) is given by
\begin{align}\CD_r(\l,0)=N_{r-2}\cupdot N_{r-5}(s_{r-3}s_{r-2}).\label{collapseD}\end{align}
The theorem follows directly from Equations (\ref{collapseD}) and the fact that $|N_k|=F_{k+1}$, whenever $k\geq 2$.

\end{proof}

\section{The $q$-analog of Kostant's weight multiplicity formula}

The $q$-analog of Kostant's weight multiplicity formula was defined in Equation (\ref{qmult}). 
We recall it here for ease of reference. The $q$-analog of Kostant's is given by
\begin{align*}
m_q(\lambda,\mu)=\displaystyle\sum_{\sigma\in W}^{}(-1)^{\ell(\sigma)}\wp_q(\sigma(\lambda+\rho)-(\mu+\rho))\label{qmult},
\end{align*}
where $\wp_q(\xi)$ denotes the $q$-analog of Kostant's partition function. This polynomial valued function is defined on $\mathfrak{h}^*$ by
$\wp_q(\xi)=c_0+c_1q+c_2q^2+c_3q^3+\cdots+c_kq^k$, where $c_i$ is the number of ways to write $\xi$ as a sum of exactly $i$ positive roots.

In this section, we let $\l=\a_1+\cdots+\a_r$ and compute $m_q(\l,0)$ in all classical Lie types. 
We begin begin by noting that type $A_r$ was considered in \cite{PH}, where combinatorial arguments were used to prove:

\begin{theorem}Let $r\geq 1$ and $\l=\a_1+\cdots+\a_r$ denote the highest root of $\mathfrak{sl}_{r+1}(\mathbb{C})$. Then 
$m_q(\l,0)=q+q^2+\cdots+q^r$.
\end{theorem}  

This provided an alternate proof of a result of Kostant related to the exponents of the Lie algebra $\mathfrak{sl}_{r+1}(\mathbb{C})$, \cite{exponents}. Type $B_r$ was considered in Theorem 3.1 of \cite{Harris}, which we state below using our notation.

\begin{theorem}Let $r\geq 2$ and $\l=\a_1+\cdots+\a_r$ denote the sum of the simple roots of $\mathfrak{so}_{2r+1}(\mathbb{C})$. Then 
$m_q(\l,0)=q^r.$
\end{theorem}

In this section we provide proofs of the analogous results for Lie types $C_r$ and $D_r$.

\begin{theorem}\label{tC}Let $r\geq 5$ and $\l=\a_1+\cdots+\a_r$ denote the sum of the simple roots of $\mathfrak{sp}_{2r}(\mathbb{C})$. Then 
$m_q(\l,0)=q+2q^2+3q^3+3q^4+\cdots+3q^{r-3}+3q^{r-2}+2q^{r-1}+q^r.$
\end{theorem}

\begin{theorem}\label{tD}Let $r\geq 7$ and $\l=\a_1+\cdots+\a_r$ denote the sum of the simple roots of $\mathfrak{so}_{2r}(\mathbb{C})$. Then
$m_q(\l,0)=q+4q^2+7q^3+8q^4+8q^5+\cdots+8q^{r-3}+7q^{r-2}+4q^{r-1}+q^r.$
\end{theorem}

Before proving Theorems \ref{tC} and \ref{tD}, we recall a couple of lemmas that will be needed in the proofs. Recall that when computing $m_q(\l,0)$ 
it suffices to evaluate the alternating sum over those elements in the collapsed Weyl alternation set $\CA_r(\l,0)$, as we noted in Section \ref{collapsed}. 
Tables \ref{lowrankC} and \ref{lowrankD} show that the theorems hold for some low rank cases in Lie Types $C_r$ and $D_r$, by giving the data associated to these cases. For every $\sigma\in\CA_r(\l,0)$, 
the tables include the expression of $\sigma(\l+\r)-\r$ as a sum of simple roots, 
as well as the value of Kostant's partition function and its $q$-analog on $\sigma(\l+\r)-\r$.

\begin{table}[H]
\centering

\begin{tabular}{|l|l|l|l|c|}
\hline \rowcolor{midgray}$\sigma\in\CA_3(\l,0)$	&$\ell(\sigma)$	& $\sigma(\l+\r)-\r$	&$\wp_q(\sigma(\l+\r)-\r)$	&$\wp_q|_{q=1}$\\

\hline
1	&0	&$\a_1+\a_2+\a_3$	&$(q+1)^2q$	&2\\ \hline
$s_2$	&1	&$\a_1+\a_3$		&$q^2$&1\\ \hline
\multicolumn{5}{|c|}{$m_q(\l,0)=q+q^2+q^3$  and  $m(\l,0)=3$}\\[1pt]
\hline
\hline\rowcolor{midgray}$\sigma\in\CA_4(\l,0)$	&$\ell(\sigma)$	& $\sigma(\l+\r)-\r$	&$\wp_q(\sigma(\l+\r)-\r)$	&$\wp_q|_{q=1}$\\

\hline
1			&0	&$\a_1+\a_2+\a_3+\a_4$	&$(q+1)^3q$		&8\\ \hline
$s_3$		&1	&$\a_1+\a_2+\a_4$	&$(q+1)q^2$		&2\\ \hline
\multicolumn{5}{|c|}{$m_q(\l,0)=q+2q^2+2q^3+q^4$  and  $m(\l,0)=6$}\\[1pt]
\hline

\end{tabular}
\caption{Data for Type $C_r$ for ranks $3$ and $4$.}
\label{lowrankC}
\end{table}

\begin{table}[H]
\centering

\begin{tabular}{|l|l|l|l|c|}

\hline \rowcolor{midgray}$\sigma\in\CA_4(\l,0)$	&$\ell(\sigma)$	& $\sigma(\l+\r)-\r$	&$\wp_q(\sigma(\l+\r)-\r)$	&$\wp_q|_{q=1}$\\

\hline
1			&0	&$\a_1+\a_2+\a_3+\a_4$	&		&8\\ \hline
$s_2$		&1	&$\a_1+\a_3+\a_4$	&		&2\\ \hline
\multicolumn{5}{|c|}{$m_q(\l,0)=q+2q^2+2q^3+q^4$  and  $m(\l,0)=6$}\\[1pt]
\hline
\hline \rowcolor{midgray}$\sigma\in\CA_5(\l,0)$	&$\ell(\sigma)$	& $\sigma(\l+\r)-\r$	&$\wp_q(\sigma(\l+\r)-\r)$	&$\wp_q|_{q=1}$\\

\hline
1				&0	&$\a_1+\a_2+\a_3+\a_4+\a_5$	&$(q+1)^4q$		&16\\ \hline
$s_3$			&1	&$\a_1+\a_2+\a_4+\a_5$	&$(q+1)^2q^2$		&4\\ \hline
\multicolumn{5}{|c|}{$m_q(\l,0)=q+3q^2+4q^3+3q^4+q^5$  and  $m(\l,0)=12$}\\[1pt]
\hline
\hline \rowcolor{midgray}$\sigma\in\CA_6(\l,0)$	&$\ell(\sigma)$	& $\sigma(\l+\r)-\r$	&$\wp_q(\sigma(\l+\r)-\r)$	&$\wp_q|_{q=1}$\\

\hline
1				&0	&$\a_1+\a_2+\a_3+\a_4+\a_5+\a_6$	&$(q+1)^5q$		&32\\ \hline
$s_2$			&1	&$\a_1+\a_3+\a_4+\a_5+\a_6$	&$(q+1)^3q^2$		&8\\ \hline
$s_4$			&1	&$\a_1+\a_2+\a_3+\a_5+\a_6$	&$(q+1)^3q^2$		&8\\ \hline
$s_2s_4$			&2	&$\a_1+\a_3+\a_5+\a_6$	&$(q+1)q^3$		&2\\ \hline
\multicolumn{5}{|c|}{$m_q(\l,0)=q+3q^2+5q^3+5q^4+3q^5+q^6$  and  $m(\l,0)=18$}\\[1pt]
\hline

\end{tabular}
\caption{Data for Type $D_r$ for ranks $ 4, \ 5 $ and $6$.}
\label{lowrankD}
\end{table}

We now recall the following lemmas that will be used in the proofs Theorems \ref{tC} and \ref{tD}.  We note that the proofs of Lemma \ref{maxC} and \ref{maxD} as well
 as Propositions \ref{wpqC} and \ref{wpqD} follow from an analogous argument as those used in the 
Type $A_r$ case found in Lemma 3.1 and Proposition 3.2 in \cite{PH}.

\begin{lemma}\label{maxC}In type $C_r$ $(r\geq 5)$:
\[\max\{\ell(\sigma):\sigma\in\C_r(\l,0)\}=
\begin{cases}
\left\lfloor\frac{r}{2}\right\rfloor&\mbox{if $\sigma$ contains $s_{r-1}s_{r-2}$ or $s_{r-2}s_{r-1}$}\\
\left\lfloor\frac{r+2}{2}\right\rfloor&\mbox{if $\sigma$ contains $s_{r-2}s_{r-1}s_{r-2}$}\\
\left\lfloor\frac{r-1}{2}\right\rfloor&\mbox{otherwise}\end{cases}\]

and
\begin{align*}
\bigg|\begin{Bmatrix}\sigma\in\C_r(\l,0):\begin{matrix}\ell(\sigma)=k+2\mbox{ and $\sigma$ contains $s_{r-1}s_{r-2}$ or $s_{r-2}s_{r-1}$}\end{matrix}\end{Bmatrix}\bigg|&=\binom{r-4-k}{k}\\
\big|\begin{Bmatrix}\sigma\in\C_r(\l,0):\begin{matrix}\ell(\sigma)=k+3\mbox{ and $\sigma$ contains $s_{r-2}s_{r-1}s_{r-2}$}\end{matrix}\end{Bmatrix}\big|&=\binom{r-4-k}{k}\\
\bigg|\begin{Bmatrix}\sigma\in\C_r(\l,0):\begin{matrix}\ell(\sigma)=k\mbox{ and $\sigma$ does not contain $s_{r-1}s_{r-2}$,}\\ \mbox{$s_{r-2}s_{r-1}$, nor $s_{r-2}s_{r-1}s_{r-2}$}\end{matrix}\end{Bmatrix}\bigg|&=\binom{r-1-k}{k}.
\end{align*}
\end{lemma}
\begin{lemma}\label{maxD}In Type $D_r$ ($r\geq 7$):

\[\max\{\ell(\sigma):\sigma\in\D_r(\l,0)\}=
\begin{cases}
\left\lfloor\frac{r-1}{2}\right\rfloor&\mbox{if $\sigma$ contains $s_{r-3}s_{r-2}$ or $s_{r-2}s_{r-3}$}\\
\left\lfloor\frac{r+1}{2}\right\rfloor&\mbox{if $\sigma$ contains $s_{r-3}s_{r-2}s_{r-3}$}\\
\left\lfloor\frac{r-3}{2}\right\rfloor&\mbox{otherwise}\end{cases}\]

and

\begin{align*}
\bigg|\begin{Bmatrix}\sigma\in\D_r(\l,0):\begin{matrix}\ell(\sigma)=k+2\mbox{ and $\sigma$ contains $s_{r-2}s_{r-3}$ or}\\ \mbox{$s_{r-3}s_{r-2}$, but not $s_{r-3}s_{r-2}s_{r-3}$}\end{matrix}\end{Bmatrix}\bigg|&=\binom{r-5-k}{k}\\
\big|\begin{Bmatrix}\sigma\in\D_r(\l,0):\ell(\sigma)=k+3\mbox{ and $\sigma$ contains $s_{r-3}s_{r-2}s_{r-3}$}\end{Bmatrix}\big|&=\binom{r-5-k}{k}\\
\bigg|\begin{Bmatrix}\sigma\in\D_r(\l,0):\begin{matrix}\ell(\sigma)=k\mbox{ and $\sigma$ does not contain }\\s_{r-2}s_{r-3},\ s_{r-3}s_{r-2},\mbox{ or }s_{r-3}s_{r-2}s_{r-3}\end{matrix}\end{Bmatrix}\bigg|&=\binom{r-3-k}{k}.
\end{align*}

\end{lemma}

\begin{proposition}\label{wpqC}Let $\sigma\in\C_r(\l,0)$ for Lie type $C_r$. Then 
\[\wp_q(\sigma(\l+\r)-\r)=\begin{cases}
q^{\ell(\sigma)}(1+q)^{r+1-2\ell(\sigma)}&\mbox{if $\sigma$ contains $s_{r-2}s_{r-1}$}\\
q^{\ell(\sigma)}(1+q)^{r-2\ell(\sigma)}&\mbox{if $\sigma$ contains $s_{r-1}s_{r-2}$}\\
q^{\ell(\sigma)-1}(1+q)^{r+2-2\ell(\sigma)}&\mbox{if $\sigma$ contains $s_{r-2}s_{r-1}s_{r-2}$}\\
q^{\ell(\sigma)+1}(1+q)^{r-1-2\ell(\sigma)}&\mbox{otherwise}.\end{cases}\]
\end{proposition}

\begin{proposition}\label{wpqD}
Let $\sigma\in\D_r(\l,0)$ for Lie type $D_r$. Then 
\[\wp_q(\sigma(\l+\r)-\r)=\begin{cases}
q^{\ell(\sigma)}(1+q)^{r+1-2\ell(\sigma)}&\mbox{ if $\sigma$ contains $s_{r-3}s_{r-2}$}\\
q^{\ell(\sigma)}(1+q)^{r-2\ell(\sigma)}&\mbox{ if $\sigma$ contains $s_{r-2}s_{r-3}$}\\
q^{\ell(\sigma)-1}(1+q)^{r+3-2\ell(\sigma)}&\mbox{ if $\sigma$ contains $s_{r-3}s_{r-2}s_{r-3}$}\\
q^{\ell(\sigma)+1}(1+q)^{r-1-2\ell(\sigma)}&\mbox{ otherwise}.\end{cases}\]
\end{proposition}

In the proofs of Theorems \ref{tC} and \ref{tD} we will make use of the following identity [Proposition 3.3 in \cite{PH}]:
\begin{proposition}\label{identity}
For $r\geq1$, $\displaystyle\sum_{k=0}^{\lfloor\frac{r-1}{2}\rfloor}(-1)^k q^{1+k}(1+q)^{r-1-2k}=\displaystyle\sum_{i=1}^{r}q^i.$
\end{proposition}

With these results in hand, we are now ready to prove the two main results of this section:
\begin{proof}[Proof of Theorem \ref{tC}]
By Equation (\ref{collapseC}) we know that 
\begin{align*}
m_q(\l,0)&=\displaystyle\sum_{\sigma\in\CC_r(\l,0)}(-1)^{\ell(\sigma)}\wp_q(\sigma(\l+\r)-\r)\\
&=\displaystyle\sum_{\sigma\in N_{r-1}}(-1)^{\ell(\sigma)}\wp_q(\sigma(\l+\r)-\r)\;\;\;\;+\;\;\;\hspace{-2mm}\displaystyle\sum_{\sigma\in N_{r-4}(s_{r-2}s_{r-1})}\hspace{-5mm}(-1)^{\ell(\sigma)}\wp_q(\sigma(\l+\r)-\r).
\end{align*}
Now by Propositions \ref{maxC}, \ref{wpqC} and \ref{identity} we have that

\[\displaystyle\sum_{\sigma\in N_{r-1}}(-1)^{\ell(\sigma)}\wp_q(\sigma(\l+\r)-\r)=\displaystyle\sum_{k=0}^{\left\lfloor\frac{r-1}{2}\right\rfloor}(-1)^{k}\binom{r-1-k}{k}q^{1+k}(1+q)^{r-1-2k}=\displaystyle\sum_{i=1}^{r}q^i\]
and
\[\displaystyle\sum_{\sigma\in N_{r-4}(s_{r-2}s_{r-1})}\hspace{-8mm}(-1)^{\ell(\sigma)}\wp_q(\sigma(\l+\r)-\r)=\displaystyle\sum_{k=0}^{\left\lfloor\frac{r-4}{2}\right\rfloor}(-1)^{k+2}\binom{r-4-k}{k}q^{2+k}(1+q)^{r-3-2k}=q(1+q)\displaystyle\sum_{i=1}^{r-3}q^i.\]

Therefore
\begin{align*}
m_q(\l,0)&=\displaystyle\sum_{i=1}^{r}q^i+q(1+q)\displaystyle\sum_{i=1}^{r-3}q^i=\displaystyle\sum_{i=1}^{r}q^i+\displaystyle\sum_{i=2}^{r-2}q^i+\displaystyle\sum_{i=3}^{r-1}q^i\\
&=q+2q^2+3q^3+3q^4+\cdots+3q^{r-3}+3q^{r-2}+2q^{r-1}+q^r.
\end{align*}
\end{proof}

\begin{proof}[Proof of Theorem \ref{tD}]
By Equation (\ref{collapseD}) we know that 
\begin{align*}
m_q(\l,0)&=\displaystyle\sum_{\sigma\in\CD_r(\l,0)}(-1)^{\ell(\sigma)}\wp_q(\sigma(\l+\r)-\r)\\
&=\displaystyle\sum_{\sigma\in N_{r-2}}(-1)^{\ell(\sigma)}\wp_q(\sigma(\l+\r)-\r)\;\;\;\;+\hspace{-2mm}\displaystyle\sum_{\sigma\in N_{r-5}(s_{r-3}s_{r-2})}\hspace{-5mm}(-1)^{\ell(\sigma)}\wp_q(\sigma(\l+\r)-\r)
\end{align*}
Now by Propositions \ref{maxD}, \ref{wpqD} and \ref{identity} we have that
\begin{align*}
\displaystyle\sum_{\sigma\in N_{r-2}}(-1)^{\ell(\sigma)}\wp_q(\sigma(\l+\r)-\r)&=\displaystyle\sum_{k=0}^{\left\lfloor\frac{r-3}{2}\right\rfloor}(-1)^{k}\binom{r-3-k}{k}q^{1+k}(1+q)^{r-1-2k}\\&=(1+q)^2\displaystyle\sum_{i=1}^{r-2}q^i,\mbox{ and}\\
\displaystyle\sum_{\sigma\in N_{r-5}(s_{r-3}s_{r-2})}\hspace{-8mm}(-1)^{\ell(\sigma)}\wp_q(\sigma(\l+\r)-\r)&=\displaystyle\sum_{k=0}^{\left\lfloor\frac{r-5}{2}\right\rfloor}(-1)^{k+2}\binom{r-5-k}{k}q^{2+k}(1+q)^{r-3-2k}\\&=q(1+q)^2\displaystyle\sum_{i=1}^{r-4}q^i.\\
\end{align*}

Therefore
\begin{align*}
m_q(\l,0)&=(1+q)^2\displaystyle\sum_{i=1}^{r-2}q^i+q(1+q)^2\displaystyle\sum_{i=1}^{r-4}q^i\\
&=\displaystyle\sum_{i=3}^{r}q^i+2\displaystyle\sum_{i=2}^{r-1}q^i+\displaystyle\sum_{i=1}^{r-2}q^i+\displaystyle\sum_{i=4}^{r-1}q^i+2\displaystyle\sum_{i=3}^{r-2}q^i+\displaystyle\sum_{i=2}^{r-3}q^i\\
&=q+4q^2+7q^3+8q^4+8q^5+\cdots+8q^{r-3}+7q^{r-2}+4q^{r-1}+q^r.
\end{align*}
\end{proof}

The results below follow from Theorems \ref{tC} and \ref{tD} and the fact that $\wp=\wp_q|_{q=1}$.
\begin{corollary}\label{tC}Let $r\geq 5$ and $\l=\a_1+\cdots+\a_r$ denote the sum of the simple roots of $\mathfrak{sp}_{2r}(\mathbb{C})$. Then 
$m(\l,0)=3r-6.$
\end{corollary}

\begin{corollary}\label{tD}Let $r\geq 7$ and $\l=\a_1+\cdots+\a_r$ denote the sum of the simple roots of $\mathfrak{so}_{2r}(\mathbb{C})$. Then
$m(\l,0)=8r-24.$
\end{corollary}

\section{Exceptional Lie algebras}\label{exceptional}
In the previous sections we have shown that if $\mathfrak{g}$ is a classical Lie algebra and $\l=\sum_{\a\in\Delta}\a$ is the sum of the simple roots, then $m(\l,0)>0$. 
We now shift our attention to the value of $m(\l,0)$ in the case where $\mathfrak{g}$ is an exceptional Lie algebra, for further information on the exceptional Lie 
algebras we point the reader to \cite{Knapp}.

\begin{namedtheorem}[Main]
Let $\mathfrak{g}$ be a simple Lie algebra over $\mathbb{C}$ and let $\l$ denote the sum of all simple roots. Then $m(\l,0)=0$ if and only if $\mathfrak{g}$ is an exceptional Lie algebra. 
\end{namedtheorem}

Our Main Theorem will follow directly from the next result:
\begin{theorem}\label{empty}
If $\mathfrak{g}$ is an exceptional Lie algebra, then $\CA(\l,0)=\emptyset.$
\end{theorem}

To prove Theorem \ref{empty} we compute the Weyl alternations sets $\A(\l,0)$, for each of the exceptional Lie algebras $G_2$, $F_4$, $E_6$, $E_7$, and $E_8$. 
Then we confirm that the collapsed Weyl alternation set is empty in each case.
We do so by using the same techniques as in the previous sections.  For the sake of brevity, and the fact that these finite calculations are analogous to those in the previous sections, 
we omit the proofs of Propositions \ref{G2}, \ref{F4}, \ref{E6}, \ref{E7}, and \ref{E8}. However, the computations in SAGE are accessible at  \url{https://sites.google.com/site/erikinskosite/hi2014code}.

\begin{proposition}\label{G2}
If $\l$ denotes the sum of the simple roots of $G_2$, then $\A(\l,0)=\{1,s_1\}$.
\end{proposition}
Proposition~\ref{G2} is part 25 of Theorem 2.6.1, in \cite{PH3}. Now observe that \[ 1(\l+\rho)-\rho=\a_1+\a_2=s_1(\l+\rho)-\rho \text{ and } \ell(1)=\ell(s_1)-1. \] 
Therefore $\CA(\l,0)=\emptyset$, which confirms that Theorem \ref{empty} holds for the Lie algebra $G_2$. We now summarize the results for the exceptional Lie algebras $F_4$, $E_6$, $E_7$, and $E_8$.

\begin{proposition}\label{F4}
If $\l$ denotes the sum of the simple roots of $F_4$, then $\A(\l,0)=\{1,s_2, s_3, s_2s_3\}$.
\end{proposition}

\begin{proposition}\label{E6}
If $\l$ denotes the sum of the simple roots of $E_6$, then \[\A(\l,0)=\{1,s_3,s_4,s_5,s_3s_4,s_4s_3,s_4s_5,s_5s_3,s_5s_4,s_3s_4s_3,s_4s_5s_4,s_5 s_3 s_4\}.\]
\end{proposition}

\begin{proposition}\label{E7}
If $\l$ denotes the sum of the simple roots of $E_7$, then \[\A(\l,0)=\left\{\begin{matrix}1,s_3,s_4,s_5,s_6,s_3 s_4,s_4 s_3,s_4 s_5,s_5 s_3,s_5 s_4,s_6 s_3,\\s_6 s_4,s_3 s_4 s_3,s_4 s_5 s_4,s_5 s_3 s_4,s_6 s_3 s_4,s_6 s_4 s_3,s_6 s_3 s_4 s_3\end{matrix}\right\}.\]
\end{proposition}

\begin{proposition}\label{E8}
If $\l$ denotes the sum of the simple roots of $E_8$, then 
\[\A(\l,0)=\left\{\begin{matrix}
1,s_3,s_4,s_5,s_6,s_7,s_3 s_4,s_4 s_3,s_4 s_5,s_5 s_3,s_5 s_4,s_6 s_3,s_6 s_4,s_7 s_3,s_7 s_4,\\
s_7 s_5,s_3 s_4 s_3,s_4 s_5 s_4,s_5 s_3 s_4,s_6 s_3 s_4,s_6 s_4 s_3,s_7 s_3 s_4,s_7 s_4 s_3,\\
s_7 s_4 s_5,s_7 s_5 s_3,s_7 s_5 s_4,s_6 s_3 s_4 s_3,s_7 s_3 s_4 s_3,s_7 s_4 s_5 s_4,s_7 s_5 s_3 s_4\end{matrix}\right\}.\]
\end{proposition}

Tables~\ref{G2table}, \ref{F4table}, \ref{E6table}, \ref{E7table}, and \ref{E8table}, list the Weyl alternation sets $\A(\l,0)$ for each exceptional Lie type. In addition, we provide the value of Kostant's partition function and its $q$-analog on $\sigma(\l+\rho)-\rho$, for all $\sigma\in\A(\l,0)$. However we note that these value are unnecessary since, based on the following observations, it is evident that $\CA(\l,0)=\emptyset$. From Tables~\ref{G2table}, \ref{F4table}, \ref{E6table}, \ref{E7table}, and \ref{E8table} note the following.
\begin{itemize}
\item In Lie type $G_2$: For every $\sigma\in\A(\l,0)$, we have that $\sigma(\l+\rho)-\rho=\sigma s_1(\l+\rho)-\rho$. 
\item In Lie type $F_4$: For every $\sigma\in\A(\l,0)$, we have that $\sigma(\l+\rho)-\rho=\sigma s_3(\l+\rho)-\rho$. 
\item In Lie types $E_6,$ $E_7$, and $E_8$: For every $\sigma\in\A(\l,0)$, we have that \[ \sigma(\l+\rho)-\rho=\sigma s_4(\l+\rho)-\rho. \]
\end{itemize}
These observations imply that Theorem \ref{empty} holds, which completes the proof of our Main Theorem.

\begin{table}[H]
\centering
\begin{tabular}{|l|l|l|l|c|} 
\hline 
\rowcolor{midgray}$\sigma\in\A(\l,0)$	&$\ell(\sigma)$	& $\sigma(\l+\r)-\r$	&$\wp_q(\sigma(\l+\r)-\r)$	&$\wp_q|_{q=1}$\\ \hline
$1$ 		&0 & $\a_1+\a_2$      &$(q+1)q$     &2\\ \hline
$s_1$ 	&1 & $\a_1+\a_2$      &$(q+1)q$     &2\\ \hline\hline
\multicolumn{5}{|c|}{$m_q(\l,0)=\sum_{\sigma\in \A(\l,0)}(-1)^{\ell(\sigma)}\wp_{q}(\sigma(\l+\rho)-\rho)=0$\mbox{ and }$m(\l,0)=0$}\\[3pt]
\hline
\end{tabular}
\caption{Data for $G_2$}
\label{G2table}
\end{table}

\begin{table}[H]
\centering
\begin{tabular}{|l|l|l|l|c|}
\hline 
\rowcolor{midgray}$\sigma\in\A(\l,0)$	&$\ell(\sigma)$	& $\sigma(\l+\r)-\r$	&$\wp_q(\sigma(\l+\r)-\r)$	&$\wp_q|_{q=1}$\\ \hline
$1$ 		&0 & $\a_1+\a_2+\a_3+\a_4$      &$(q+1)^3q$     &8\\ \hline
$s_3$ 	&1 & $\a_1+\a_2+\a_3+\a_4$      &$(q+1)^3q$     &8\\ \hline 
\rowcolor{lightgray}$s_2$ 	&1 & $\a_1+\a_3+\a_4$      &$(q+1)q^2$			  &2\\ \hline
\rowcolor{lightgray}$s_2s_3$ &2 & $\a_1+\a_3+\a_4$      &$(q+1)q^2$     			 &2\\ \hline\hline
\multicolumn{5}{|c|}{$m_q(\l,0)=\sum_{\sigma\in \A(\l,0)}(-1)^{\ell(\sigma)}\wp_{q}(\sigma(\l+\rho)-\rho)=0$\mbox{ and }$m(\l,0)=0$}\\[3pt]
\hline
\end{tabular}
\caption{Data for $F_4$}
\label{F4table}
\end{table}

\begin{table}[H]
\centering
\begin{tabular}{|l|l|l|l|c|}
\hline \rowcolor{midgray}$\sigma\in\A(\l,0)$	&$\ell(\sigma)$	& $\sigma(\l+\r)-\r$	&$\wp_q(\sigma(\l+\r)-\r)$	&$\wp_q|_{q=1}$\\ \hline
$1$ 				&0 & $\a_1+\a_2+\a_3+\a_4+\a_5+\a_6$      &$(q+1)^5q$   	  	&32\\ \hline
$s_4$ 			&1 & $\a_1+\a_2+\a_3+\a_4+\a_5+\a_6$      &$(q+1)^5q$	     		&32\\ \hline
\rowcolor{lightgray}$s_3$ 			&1 & $\a_1+\a_2+\a_4+\a_5+\a_6$      &$(q+1)^3q^2$     					&8\\ \hline
\rowcolor{lightgray}$s_3s_4$	 		&2 & $\a_1+\a_2+\a_4+\a_5+\a_6$      &$(q+1)^3q^2$     					&8\\ \hline
$s_5$ 			&1 & $\a_1+\a_2+\a_3+\a_4+\a_6$      &$(q+1)^3q^2$     					&8\\ \hline
$s_5s_4$ 			&2 & $\a_1+\a_2+\a_3+\a_4+\a_6$      &$(q+1)^3q^2$     					&8\\ \hline
\rowcolor{lightgray}$s_4s_3$ 			&2 & $\a_1+\a_2+\a_5+\a_6$      &$(q+1)q^3$     							&2\\ \hline
\rowcolor{lightgray}$s_4s_3s_4$ 		&3 & $\a_1+\a_2+\a_5+\a_6$      &$(q+1)q^3$     							&2\\ \hline

$s_4s_5$ 			&2 & $\a_1+\a_2+\a_3+\a_6$      &$(q+1)q^3$     							&2\\ \hline
$s_4s_5s_4$ 		&3 & $\a_1+\a_2+\a_3+\a_6$      &$(q+1)q^3$     							&2\\ \hline
\rowcolor{lightgray}$s_5s_3$ 			&2 & $\a_1+\a_2+\a_4+\a_6$      &$(q+1)q^3$     							&2\\ \hline
\rowcolor{lightgray}$s_5 s_3 s_4$ 		&3 & $\a_1+\a_2+\a_4+\a_6$      &$(q+1)q^3$     							&2\\ \hline
\hline
\multicolumn{5}{|c|}{$m_q(\l,0)=\sum_{\sigma\in \A(\l,0)}(-1)^{\ell(\sigma)}\wp_{q}(\sigma(\l+\rho)-\rho)=0$\mbox{ and }$m(\l,0)=0$}\\[3pt]
\hline
\end{tabular}
\caption{Data for $E_6$}
\label{E6table}
\end{table}

\begin{table}[H]
\centering
\begin{tabular}{|l|l|l|l|c|}
\hline \rowcolor{midgray}$\sigma\in\A(\l,0)$	&$\ell(\sigma)$	& $\sigma(\l+\r)-\r$	&$\wp_q(\sigma(\l+\r)-\r)$	&$\wp_q|_{q=1}$\\ \hline
$1$ 				&0 & $\a_1+\a_2+\a_3+\a_4+\a_5+\a_6+\a_7$      &$(q+1)^6q$     	&64\\ \hline
$s_4$ 			&1 & $\a_1+\a_2+\a_3+\a_4+\a_5+\a_6+\a_7$      &$(q+1)^6q$     	&64\\ \hline
\rowcolor{lightgray}$s_3$ 			&1 & $\a_1+\a_2+\a_4+\a_5+\a_6+\a_7$      &$(q+1)^4q^2$   &16\\ \hline
\rowcolor{lightgray}$s_3 s_4$ 		&2 & $\a_1+\a_2+\a_4+\a_5+\a_6+\a_7$      &$(q+1)^4q^2$   &16\\ \hline
$s_5$ 			&1 & $\a_1+\a_2+\a_3+\a_4+\a_6+\a_7$      &$(q+1)^4q^2$   &16\\ \hline
$s_5 s_4$ 		&2 & $\a_1+\a_2+\a_3+\a_4+\a_6+\a_7$      &$(q+1)^4q^2$    &16\\ \hline
\rowcolor{lightgray}$s_6$ 			&1 & $\a_1+\a_2+\a_3+\a_4+\a_5+\a_7$      &$(q+1)^4q^2$   &16\\ \hline
\rowcolor{lightgray}$s_6 s_4$ 		&2 & $\a_1+\a_2+\a_3+\a_4+\a_5+\a_7$      &$(q+1)^4q^2$    &16\\ \hline
$s_4 s_3$ 		&2 & $\a_1+\a_2+\a_5+\a_6+\a_7$      &$(q+1)^2q^3$   &4\\ \hline
$s_4 s_3 s_4$ 		&3 & $\a_1+\a_2+\a_5+\a_6+\a_7$      &$(q+1)^2q^3$    &4\\ \hline
\rowcolor{lightgray}$s_4 s_5$	        &2 & $\a_1+\a_2+\a_3+\a_6+\a_7$      &$(q+1)^2q^3$   &4\\ \hline
\rowcolor{lightgray}$s_4 s_5 s_4$ 		&3 & $\a_1+\a_2+\a_3+\a_6+\a_7$      &$(q+1)^2q^3$    &4\\ \hline
$s_5 s_3$ 		&2 & $\a_1+\a_2+\a_4+\a_6+\a_7$      &$(q+1)^2q^3$   &4\\ \hline
$s_5 s_3 s_4$ 		&3 & $\a_1+\a_2+\a_4+\a_6+\a_7$      &$(q+1)^2q^3$    &4\\ \hline
\rowcolor{lightgray}$s_6 s_3$ 		&2 & $\a_1+\a_2+\a_4+\a_5+\a_7$      &$(q+1)^2q^3$   &4\\ \hline
\rowcolor{lightgray}$s_6 s_3 s_4$ 		&3 & $\a_1+\a_2+\a_4+\a_5+\a_7$      &$(q+1)^2q^3$    &4\\ \hline
$s_6 s_4 s_3$ 		&3 & $\a_1+\a_2+\a_5+\a_7$      &$q^4$     		  &1\\ \hline
$s_6 s_4 s_3 s_4$ 	&4 & $\a_1+\a_2+\a_5+\a_7$      &$q^4$     		  &1\\ \hline\hline
\multicolumn{5}{|c|}{$m_q(\l,0)=\sum_{\sigma\in \A(\l,0)}(-1)^{\ell(\sigma)}\wp_{q}(\sigma(\l+\rho)-\rho)=0$\mbox{  and  }$m(\l,0)=0$}\\[3pt]
\hline
\end{tabular}
\caption{Data for $E_7$}
\label{E7table}
\end{table}

\begin{table}[H]
\centering
\begin{tabular}{|l|l|l|l|l|}
\hline \rowcolor{midgray}$\sigma\in\A(\l,0)$	&$\ell(\sigma)$	& $\sigma(\l+\r)-\r$	&$\wp_q(\sigma(\l+\r)-\r)$	&$\wp_q|_{q=1}$\\ \hline
$1$ 				&0 & $\a_1+\a_2+\a_3+\a_4+\a_5+\a_6+\a_7+\a_8$      &$(q+1)^7q$		&128\\ \hline
$s_4$ 			&1 & $\a_1+\a_2+\a_3+\a_4+\a_5+\a_6+\a_7+\a_8$      &$(q+1)^7q$     		&128\\ \hline
\rowcolor{lightgray}$s_3$ 			&1 & $\a_1+\a_2+\a_4+\a_5+\a_6+\a_7+\a_8$      &$(q+1)^5q^2$		&32\\ \hline
\rowcolor{lightgray}$s_3 s_4$ 		&2 & $\a_1+\a_2+\a_4+\a_5+\a_6+\a_7+\a_8$      &$(q+1)^5q^2$     	&32\\ \hline
$s_5$ 			&1 & $\a_1+\a_2+\a_3+\a_4+\a_6+\a_7+\a_8$      &$(q+1)^5q^2$     	&32\\ \hline
$s_5 s_4$ 		&2 & $\a_1+\a_2+\a_3+\a_4+\a_6+\a_7+\a_8$      &$(q+1)^5q^2$     	&32\\ \hline
\rowcolor{lightgray}$s_6$ 			&1 & $\a_1+\a_2+\a_3+\a_4+\a_5+\a_7+\a_8$      &$(q+1)^5q^2$     	&32\\ \hline
\rowcolor{lightgray}$s_6 s_4$ 		&2 & $\a_1+\a_2+\a_3+\a_4+\a_5+\a_7+\a_8$      &$(q+1)^5q^2$     	&32\\ \hline
$s_7$ 			&1 & $\a_1+\a_2+\a_3+\a_4+\a_5+\a_6+\a_8$      &$(q+1)^5q^2$     	&32\\ \hline
$s_7 s_4$ 		&2 & $\a_1+\a_2+\a_3+\a_4+\a_5+\a_6+\a_8$      &$(q+1)^5q^2$     	&32\\ \hline
\rowcolor{lightgray}$s_4 s_3$ 		&2 & $\a_1+\a_2+\a_5+\a_6+\a_7+\a_8$      &$(q+1)^5q^2$     	&32\\ \hline
\rowcolor{lightgray}$s_4 s_3 s_4$ 		&3 & $\a_1+\a_2+\a_5+\a_6+\a_7+\a_8$      &$(q+1)^3q^3$     	&8\\ \hline
$s_4 s_5$ 		&2 & $\a_1+\a_2+\a_3+\a_6+\a_7+\a_8$      &$(q+1)^3q^3$     	&8\\ \hline
$s_4 s_5 s_4$ 		&3 & $\a_1+\a_2+\a_3+\a_6+\a_7+\a_8$      &$(q+1)^3q^3$     	&8\\ \hline
\rowcolor{lightgray}$s_5 s_3$ 		&2 & $\a_1+\a_2+\a_4+\a_6+\a_7+\a_8$      &$(q+1)^3q^3$     	&8\\ \hline
\rowcolor{lightgray}$s_5 s_3 s_4$ 		&3 & $\a_1+\a_2+\a_4+\a_6+\a_7+\a_8$      &$(q+1)^3q^3$     	&8\\ \hline
$s_6 s_3$ 		&2 & $\a_1+\a_2+\a_4+\a_5+\a_7+\a_8$      &$(q+1)^3q^3$     	&8\\ \hline
$s_6 s_3 s_4$ 		&3 & $\a_1+\a_2+\a_4+\a_5+\a_7+\a_8$      &$(q+1)^3q^3$     	&8\\ \hline
\rowcolor{lightgray}$s_7 s_3$ 		&2 & $\a_1+\a_2+\a_4+\a_5+\a_6+\a_8$      &$(q+1)^3q^3$     	&8\\ \hline
\rowcolor{lightgray}$s_7 s_3 s_4$ 		&3 & $\a_1+\a_2+\a_4+\a_5+\a_6+\a_8$      &$(q+1)^3q^3$     	&8\\ \hline
$s_7 s_5$ 		&2 & $\a_1+\a_2+\a_3+\a_4+\a_6+\a_8$      &$(q+1)^3q^3$     	&8\\ \hline
$s_7 s_5 s_4$ 		&3 & $\a_1+\a_2+\a_3+\a_4+\a_6+\a_8$      &$(q+1)^3q^3$     	&8\\ \hline
\rowcolor{lightgray}$s_6 s_4 s_3$ 		&3 & $\a_1+\a_2+\a_5+\a_7+\a_8$      &$(q+1)q^4$     	&2\\ \hline
\rowcolor{lightgray}$s_6 s_4 s_3 s_4$ 	&4 & $\a_1+\a_2+\a_5+\a_7+\a_8$      &$(q+1)q^4$     	&2\\ \hline
$s_7 s_4 s_3$ 		&3 & $\a_1+\a_2+\a_5+\a_6+\a_8$      &$(q+1)q^4$     	&2\\ \hline
$s_7 s_4 s_3 s_4$ 	&4 & $\a_1+\a_2+\a_5+\a_6+\a_8$      &$(q+1)q^4$     	&2\\ \hline
\rowcolor{lightgray}$s_7 s_4 s_5$ 		&3 & $\a_1+\a_2+\a_3+\a_6+\a_8$      &$(q+1)q^4$     	&2\\ \hline
\rowcolor{lightgray}$s_7 s_4 s_5 s_4$ 	&4 & $\a_1+\a_2+\a_3+\a_6+\a_8$      &$(q+1)q^4$     	&2\\ \hline
$s_7 s_5 s_3$ 		&3 & $\a_1+\a_2+\a_4+\a_6+\a_8$      &$(q+1)q^4$     	&2\\ \hline
$s_7 s_5 s_3 s_4$ 	&4 & $\a_1+\a_2+\a_4+\a_6+\a_8$      &$(q+1)q^4$     	&2\\ \hline
\hline
\multicolumn{5}{|c|}{$m_q(\l,0)=\sum_{\sigma\in \A(\l,0)}(-1)^{\ell(\sigma)}\wp_{q}(\sigma(\l+\rho)-\rho)=0$\mbox{   and   }$m(\l,0)=0$}\\[3pt]
\hline
\end{tabular}
\caption{Data for $E_8$}
\label{E8table}
\end{table}

\pagebreak

\section*{Acknowledgements}
The authors would like to thank Drs. Nicholas Teff and Jeb Willenbring for many helpful conversations, which helped in the development and writing of this paper. Pamela E. Harris gratefully acknowledges travel support from the Photonics Research Center at the United States Military Academy.

\end{document}